\documentclass[letterpaper,conference]{ieeeconf}

\IEEEoverridecommandlockouts

\usepackage{amssymb}
\usepackage{graphicx}
\usepackage{algorithm,algorithmic}	
\usepackage{amsmath}
\usepackage[compress]{cite}
\usepackage{tikz}
\usepackage{url}
\usepackage{breakurl}

\DeclareRobustCommand\particle{\raisebox{1.3pt}{\tikz{\draw[-,black,dotted,line width = 0.9pt](0,0) -- (5mm,0);}}}

\DeclareRobustCommand\approximation{\raisebox{1.3pt}{\tikz{\draw[-,red,solid,line width = 0.9pt](0,0) -- (5mm,0);}}}

\DeclareRobustCommand\ole{\raisebox{1.3pt}{\tikz{\draw[-,blue,dashed,line width = 0.9pt](0,0) -- (5mm,0);}}}

\DeclareRobustCommand\map{\raisebox{0pt}{\tikz{\draw[-,black!30!green,solid,line width = 0.9pt](0,0) -- (5mm,0);\node[circle,black!30!green,line width = 0.9pt,minimum size=1.3mm,draw,inner sep=0pt] at (.25,0) {};}}}

\renewcommand{\baselinestretch}{0.985}

\newtheorem{problem}{Problem}

\newtheorem{lemma}{Lemma}

\newtheorem{theorem}{Theorem}
\newtheorem{corollary}{Corollary}

\newcommand{\diag}{\mathop{\text{diag}}}
\newcommand{\argmin}{\mathop{\text{argmin}}}

\newcommand{\trace}{\mathrm{Tr}}

\usepackage{bbm}

\title{Optimal State Estimation with Measurements Corrupted by Laplace Noise\thanks{The work of F. Farokhi was supported by a McKenzie Fellowship, ARC grant LP130100605, an early career grant from the Melbourne School of Engineering.}\thanks{The work of H. Sandberg and J. Milosevic was supported by the Swedish Civil Contingencies Agency through the CERCES project.}\thanks{The authors would like to thank Cristian R. Rojas for discussions.}}
\author{Farhad Farokhi\thanks{F. Farokhi is with the Department of Electrical and Electronic Engineering, the University Melbourne, Parkville, VIC 3010, Australia. Email:ffarokhi@unimelb.edu.au}, Jezdimir Milosevic, Henrik Sandberg\thanks{H. Sandberg and  J. Milosevic are with the Department of Automatic Control, KTH Royal Institute of Technology, Stockholm, Sweden. Emails:\{jezdimir,hsan\}@kth.se}\vspace{-.1in}}

\begin{document}
\maketitle

\begin{abstract} Optimal state estimation for linear discrete-time systems is considered. Motivated by the literature on differential privacy, the measurements are assumed to be corrupted by Laplace noise. The optimal least mean square error estimate of the state is approximated using a randomized method. The method relies on that the Laplace noise can be rewritten as Gaussian noise scaled by Rayleigh random variable. The probability of the event that the distance between the approximation and the best estimate is smaller than a constant is determined as function of the number of parallel Kalman filters that is used in the randomized method. This estimator is then compared with the optimal linear estimator, the maximum \textit{a posteriori} (MAP) estimate of the state, and the particle filter. 
\end{abstract}

\section{Introduction} \label{sec:intro}
Differential privacy provides a systematic approach for responding to statistical queries on stochastic databases while preserving the privacy of the individuals~\cite{Dwork2008}. Most often the outcome of the query is corrupted by Laplace noise where its scaling parameter is proportional to the sensitivity of the outcome to individual changes in the entries of the database. This ensures that the changes in the entries do not create pronounced variations in the response (in terms of the probability of observing various responses) and thus reverse engineering techniques cannot be used. More recently, the notion of differential privacy has been extended to dynamical systems~\cite{dwork2010differential,le2014differentially, huang2012differentially,huang2014cost}. In~\cite{le2014differentially}, the notion of differential privacy was extended to dynamical systems to preserve the privacy of the inputs. Iterative consensus seeking algorithms that preserve the privacy of the initial conditions of the participants were considered in~\cite{huang2012differentially}. Corrupting the state measurements of agents in a distributed control setup was explored in~\cite{huang2014cost} to keep the preferences of the agents (i.e., the way points that need to be followed) private. All the mentioned studies propose adding Laplace noise to the transmitted measurements or the system dynamics. Therefore, for controlling these systems, optimal state estimation techniques are required that can handle scenarios where the process noise and/or the measurement noise follow the Laplace distribution. 

In this paper, specifically, optimal state estimation for linear discrete-time systems when the measurements are corrupted by the Laplace noise is considered. The optimal least mean square error estimate of the state (which is equal to the conditional mean of the state given the output measurements) is approximated using a randomized algorithm. To do so, the Laplace noise is written as Gaussian noise scaled by Rayleigh random variable. Then, several time-varying Kalman filters are run in parallel to generate the best estimate for various choices of Rayleigh scales (that are selected randomly). The outcomes of the Kalman filters are averaged to construct an estimate. The probability of the event that the distance between the approximation and the best estimate is bounded by a constant is derived. This probability is a function of the number of the parallel-running Kalman filters and the error bound. Therefore, the performance of this filter can be arbitrarily improved at the expense of the computational expense of running several Kalman filters. The performance of this filter is compared with the optimal linear estimator, the maximum \textit{a posteriori} (MAP) estimate of the state, and the particle filter. 

Non-Gaussian estimation problems have been studied extensively in the past~\cite{gordon1993novel,arulampalam2002tutorial}. Most often particle filters, which belong to the family of randomized methods, are utilized to approximate the conditional density function of the state given the measurements. Subsequently, the best estimate is constructed using this conditional density function. However, to the best of our knowledge, the special structure of the Laplace random variables, i.e., the fact that a Laplace noise can be written as a Gaussian noise scaled by a Rayleigh random variable, has not been utilized to develop more efficient filters. This avenue is explored in this paper.

The MAP estimator in this paper is similar, in essence, to least absolute shrinkage and selection operator (LASSO), which is a regression analysis method that performs both variable selection and regularization at the same time~\cite{tibshirani1996regression}. Note that this approach was utilized earlier in~\cite{carrillo2015state} to state estimation with impulsive noises. The impulsive noise was modelled using Laplace noise and the same filter was developed. The static version of the filter for measurements corrupted by Laplace noise was more recently used in~\cite{7402921}. 

The rest of the paper is organized as follows. First, this section is concluded by presenting  notations.  The problem formulation is introduced in Section~\ref{sec:problem}. A randomized algorithm is developed in Section~\ref{sec:randomized} to approximate the least mean square error estimate of the state. The optimal linear estimators are designed in Section~\ref{sec:linear}. Section~\ref{sec:MLE} constructs the MAP estimate of the state. The numerical examples are presented in Section~\ref{sec:numerical} and Section~\ref{sec:conc} concludes the paper.

\subsection{Notation}
The sets of natural, integer, and real numbers are respectively denoted by $\mathbb{N}$, $\mathbb{Z}$, and $\mathbb{R}$. Let $\mathbb{N}_0=\mathbb{N}\cup\{0\}$. The set of non-negative real numbers is denoted by $\mathbb{R}_{\geq 0}=\{x\in\mathbb{R}\,|\,x\geq 0\}$. The set of positive definite matrices in $\mathbb{R}^{n\times n}$ is denoted by $\mathcal{S}_{+}^n$.

A random variable $w\in\mathbb{R}^n$, for some $n\in\mathbb{N}$, is said to follow the Gaussian distribution with mean $\mu\in\mathbb{R}^n$ and variance $W\in\mathcal{S}_{+}^n$, or equivalently $x\sim \mathcal{N}(\mu,W)$, if 
\begin{align*}
\mathbb{P}\{w\in\mathcal{W}\}=\int_{w'\in\mathcal{W}} &\bigg[\frac{1}{\sqrt{(2\pi)^n\det(W)}}\\&\hspace{-.36in}\times \exp\bigg(-\frac{1}{2}(x-\mu)^\top W^{-1}(x-\mu) \bigg) \bigg]\mathrm{d}w'
\end{align*}
for any Lebesgue-measurable set $\mathcal{W}\subseteq \mathbb{R}^n$. A scalar random variable $v\in\mathbb{R}$ is said to follow the Laplace distribution with mean $a\in\mathbb{R}$ and scaling parameter $b\in\mathbb{R}_{\geq 0}$, or equivalently $v\sim \mathcal{L}(a,b)$, if 
\begin{align*}
\mathbb{P}\{v\in\mathcal{V}\}=\int_{v'\in\mathcal{V}}\bigg[ \frac{1}{2b}\exp\left(-\frac{|v'-a|}{b}\right)\bigg]\mathrm{d}v'
\end{align*}
for any Lebesgue-measurable set $\mathcal{V}\subseteq\mathbb{R}$. Evidently, if $v\sim\mathcal{L}(0,b)$, then $\mathbb{E}\{v^2\}=2b^2$. A scalar random variable $x$ follows the Rayleigh distribution with scale parameter $\theta$, or equivalently $x\sim \mathcal{R}(\theta)$, if 
\begin{align*}
\mathbb{P}\{x\in\mathcal{X}\}=\int_{x'\in\mathcal{X}} \frac{x'}{\theta^2}\exp\left(\frac{-x'^2}{2\theta^2}\right)\mathrm{d}x'
\end{align*}
for any Lebesgue-measurable set $\mathcal{X}\subseteq\mathbb{R}_{\geq 0}$.

\section{Problem Formulation} \label{sec:problem}
Consider the discrete-time linear dynamical system 
\begin{subequations}\label{eqn:dynamics}
\begin{align}
x[k+1]&=A[k]x[k]+w[k],\,x[0]=x_0,\\
y[k]&=C[k]x[k]+v[k]
\end{align}
\end{subequations}
where $x[k]\in\mathbb{R}^n$ is the state, $w[k]\in\mathbb{R}^n$ is the process noise, $y[k]\in\mathbb{R}^p$ is the output, and $v[k]\in\mathbb{R}^p$ is the measurement noise. Assume that $(w[k])_{k\in\mathbb{N}_0}$ is a sequence of i.i.d.\footnote{i.i.d. stands for identically and independently distributed.} zero-mean Gaussian random variables with variance $W[k]$ at time $k\in\mathbb{N}_0$. Equivalently, $w[k]\sim\mathcal{N}(0,W[k])$ for all $k\in\mathbb{N}_0$. Motivated by the setup of differential privacy in which the measurements are often corrupted by Laplace noise to preserve the privacy of the individuals (see Section~\ref{sec:intro}), $(v[k])_{k\in\mathbb{N}_0}$ is assumed to be a sequence of i.i.d. Laplace random variables with zero mean and variance $V[k]$ at time $k\in\mathbb{N}_0$. Note that $v_i[k]$ is assumed to be statistically independent of $v_j[k]$ if $i\neq j$. Thus, $v_i[k]\sim \mathcal{L}(0,\sqrt{V_{ii}[k]/2})$. Assume that $\mathbb{E}\{x_0\}=0$ and $\mathbb{E}\{x_0x_0^\top\}=X_0\in\mathcal{S}_+^n$. The assumption that $\mathbb{E}\{x_0\}=0$ can always be satisfied by a simple change of variable and is thus without the loss of generality. The following problem is addressed in this paper.

\begin{problem} \label{prob:1} Find the state estimate $(\hat{x}[k])_{k\in\mathbb{N}_0}$ such that (\textit{i})~$\mathbb{E}\{\hat{x}[k]\}=\mathbb{E}\{x[k]\}=0$ and (\textit{ii})~the variance of error $\mathbb{E}\{\|x-\hat{x}\|_2^2\}=\trace(\mathbb{E}\{(x[k]-\hat{x}[k])(x[k]-\hat{x}[k])^\top\})$ is minimized. 
\end{problem}

The solution of this problem is the conditional exception $\mathbb{E}\{x[k]\,|\,y[0],\dots,y[k]\}$, which can be realized using a finite-order linear system in the case where the measurement noise has a Gaussian distribution. This system is the celebrated Kalman filter. However, for more general distributions, the least mean square estimate of the state $\mathbb{E}\{x[k]\,|\,y[0],\dots,y[k]\}$ is, in general, not  realizable by finite-order linear systems. In this paper, the state estimate is approximated using a randomized method and its performance is compared with that of an optimal linear estimators and the MAP estimator. Note that MAP is not necessarily an unbiased filter (which is a criteria of Problem~\ref{prob:1}).

\section{Approximating the Optimal Estiamte} \label{sec:randomized}
In this section, a randomized method is proposed to generate a state estimate that is close to the optimal least mean square error estimate of the  state, which is equal to $\mathbb{E}\{x[k]\,|\,(y[t])_{t=0}^k\}$. This filter uses the observation that  the Laplace noise can be replaced with an appropriately scaled Gaussian noise. 

\begin{lemma} For $b>0$, let $\tau\sim\mathcal{R}(b)$ and $v\sim\mathcal{N}(0,\tau^2)$. Then, $v\sim\mathcal{L}(0,b)$. 
\end{lemma}

\begin{proof}  Note that, Lebesgue-measurable set $\mathcal{V}\subseteq\mathbb{R}$, we get
\begin{align}
\mathbb{P}\{v\in\mathcal{V}\}=\hspace{-.04in}
&\int_{v'\in\mathcal{V}}\int_{\tau'=0}^{+\infty} \frac{1}{\sqrt{2\pi}\tau'}\exp\hspace{-.04in}\left(\frac{-v'^2}{2\tau'^2} \right)\hspace{-.04in}\nonumber\\
&\hspace{.8in}\times\frac{\tau'}{b^2}\exp\hspace{-.04in}\left(\frac{-\tau'^2}{2b^2}\right)\hspace{-.04in}\mathrm{d}\tau'\mathrm{d}v'\nonumber\\=\hspace{-.04in}
&\int_{v'\in\mathcal{V}}\int_{\tau'=0}^{+\infty} \frac{1}{\sqrt{2\pi}b^2}\exp\left(\hspace{-.04in}-\frac{v'^2}{2\tau'^2}\hspace{-.02in}-\hspace{-.02in}\frac{\tau'^2}{2b^2}\hspace{-.04in}\right)\hspace{-.04in}\mathrm{d}\tau'\mathrm{d}v'\nonumber\\=\hspace{-.04in}
&\int_{v'\in\mathcal{V}} \frac{1}{2b} \exp\left(-2\sqrt{\frac{v'^2}{4b^2}}\right)\mathrm{d}v'\label{eqn:proof:equality}
\\=\hspace{-.04in}
&\int_{v'\in\mathcal{V}} \frac{1}{2b} \exp\left(-\frac{|v'|}{b}\right)\mathrm{d}v', \nonumber
\end{align}
where~\eqref{eqn:proof:equality} follows from (5.10.10) in~\cite[p.\,118]{lebedev1972special}.
This shows that $v\sim\mathcal{L}(0,b)$. 
\end{proof}

Following this observation, the output equation of the underlying system can be redefined to be
\begin{align*}
y[k]=Cx[k]+\bar{v}[k],
\end{align*}
where $(\bar{v}[k])_{k\in\mathbb{N}_0}$ denotes  a sequence of i.i.d. zero-mean Gaussian random variables with variance $T[k]=\diag(\tau_{1}^2[k],\dots,\tau_{p}^2[k])$ where $\tau_{i}[k]\sim\mathcal{R}(\sqrt{V_{ii}[k]/2})$ for all $1\leq i\leq p$. Now, note that
\begin{align*}
\hat{x}[k]
&=\mathbb{E}\{x[k]\,|\,(y[t])_{t=0}^k\}\\
&=\mathbb{E}\{\mathbb{E}\{x[k]\,|\,(y[t])_{t=0}^k,(T[t])_{t=0}^k\}\,|\,(y[t])_{t=0}^k\}.
\end{align*}
The estimates {$\mathbb{E}\{x[k]|(y[t])_{t=0}^k,(T[t])_{t=0}^k\}$ are easy to compute. In fact, a time-varying Kalman filter can be used to calculate $\mathbb{E}\{x[k]|(y[t])_{t=0}^k,(T[t])_{t=0}^k\}$ because, for a fixed sequence of matrices $(T[t])_{t=0}^k$, the problem can be translated into optimal filtering of a discrete-time linear time-varying system with Gaussian process and observation noises. For all $0\leq \ell\leq k$, define
\begin{align*}
\tilde{P}[\ell]&=\mathbb{E}\{(x[\ell]-\mathbb{E}\{x[\ell]|(y[t])_{t=0}^\ell,(T[t])_{t=0}^\ell\})\\
&\times(x[\ell]-\mathbb{E}\{x[\ell]|(y[t])_{t=0}^\ell,(T[t])_{t=0}^\ell\})^\top\,|\,(T[t])_{t=0}^\ell\}.
\end{align*}
Following the Kalman filter construction results in
\begin{align*}
\mathbb{E}\{x[\ell]|&(y[t])_{t=0}^{\ell-1},(T[t])_{t=0}^{\ell-1}\}\\&=A[\ell]\mathbb{E}\{x[\ell-1]|(y[t])_{t=0}^{\ell-1},(T[t])_{t=0}^{\ell-1}\},
\end{align*}
and
\begin{align*}
\mathbb{E}\{x[\ell]|&(y[t])_{t=0}^\ell,(T[t])_{t=0}^\ell\}=\mathbb{E}\{x[\ell]|(y[t])_{t=0}^{\ell-1},(T[t])_{t=0}^{\ell-1}\}\\
&+L[\ell](y[\ell]-C[\ell]\mathbb{E}\{x[\ell]|(y[t])_{t=0}^{\ell-1},(T[t])_{t=0}^{\ell-1}\}),
\end{align*}
where
\begin{align*}
L[\ell]=\hat{P}[\ell]C[\ell]^\top(C[\ell]\hat{P}[\ell]C[\ell]^\top+T[\ell])^{-1},
\end{align*}
and
\begin{align*}
\hat{P}[\ell]=
&\mathbb{E}\{(x[\ell]-\mathbb{E}\{x[\ell]|(y[t])_{t=0}^{\ell-1},(T[t])_{t=0}^{\ell-1}\})\\
&\hspace{-.05in}\times(x[\ell]-\mathbb{E}\{x[\ell]|(y[t])_{t=0}^{\ell-1},(T[t])_{t=0}^{\ell-1}\})^\top\,|\,(T[t])_{t=0}^{\ell}\}\\
=&A[\ell]^\top \tilde{P}[\ell-1]A[\ell]+W[\ell].
\end{align*}
It can also be observed that
\begin{align*}
\tilde{P}[\ell]
=&(I-L[\ell]C[\ell])\hat{P}[\ell]\\
=&\hat{P}[\ell]-\hat{P}[\ell]C[\ell]^\top(C[\ell]\hat{P}[\ell]C[\ell]^\top+T[\ell])^{-1}C[\ell]\hat{P}[\ell].
\end{align*}
The filter is initialized by 
$\mathbb{E}\{x[0]|(y[t])_{t=0}^{-1},(T[t])_{t=0}^{-1}\}=\mathbb{E}\{x[0]\}=\mathbb{E}\{x_0\}=0,$
and
\begin{align*}
\hat{P}[0]=&\mathbb{E}\{(x[0]-\mathbb{E}\{x[0]|(y[t])_{t=0}^{-1},(T[t])_{t=0}^{-1}\})\\
&\hspace{-.05in}\times(x[0]-\mathbb{E}\{x[0]|(y[t])_{t=0}^{-1},(T[t])_{t=0}^{-1}\})^\top\,|\,(T[t])_{t=0}^{-1}\}\\
=&\mathbb{E}\{x_0x_0^\top\}=X_0.
\end{align*}
Notice that, in this derivation, both $\tilde{P}[k]$ and $\hat{P}[k]$ are random variables depending on the random variable $(T[t])_{t=0}^{k}$. 

Now, at each iteration $k$, we may select samples $T^i[k]$, $1\leq i\leq I$, from the conditional density $p(T[k]|(y[t])_{t=0}^k)$. Following the recipe above, $\mathbb{E}\{x[k]|(y[t])_{t=0}^k,(T^i[t])_{t=0}^k\}$ can be computed easily. Define
\begin{align*}
\hat{x}^{\mathrm{appx}}[k]=\frac{1}{I}\sum_{i=1}^I \mathbb{E}\{x[k]|(y[t])_{t=0}^k,(T^i[t])_{t=0}^k\}.
\end{align*}
The difficulty of this algorithm is to generate samples from $p(T[k]|(y[t])_{t=0}^k)$. This is discussed in detail towards the end of this section. However, first, it is shown that this randomized algorithm can generate arbitrarily close approximations of the least mean square error estimate of the state.

\begin{theorem} \label{tho:3} There exists $M[k]>0$ such that $\mathbb{P}\{\|\hat{x}[k]-\hat{x}^{\mathrm{appx}}[k]\|_2\leq \epsilon\}\geq 1-\delta$ for $I=M[k]\delta^{-1}\epsilon^{-2}$.
\end{theorem}

\begin{proof} Note that $\hat{x}[k]=\mathbb{E}\{\hat{x}^{\mathrm{appx}}|(y[t])_{t=0}^k\}$. The Chebyshev's inequality~\cite[p.\,446-451]{laha1979probability} gives
\begin{align*}
\mathbb{P}&\left\{\left\|\hat{x}^{\mathrm{appx}}[k]-\hat{x}[k]\right\|_2\geq \epsilon\right\}\leq\mathbb{E}\left\{\frac{\|\hat{x}^{\mathrm{appx}}[k]-\hat{x}[k]\|_2^2}{\epsilon^2}\right\}.
\end{align*}
For all $i\neq j$, because of the statistical independence of $(T^i[t])_{t=0}^k$ and $(T^j[t])_{t=0}^k$, it can be deduced that
\begin{align*}
&\mathbb{E}\left\{\left(\mathbb{E}\{x[k]|(y[t])_{t=0}^k,(T^i[t])_{t=0}^k\}\hspace{-.03in}-\hspace{-.03in}\mathbb{E}\{x[k]|(y[t])_{t=0}^k\}\right)\right.\\
&\times\hspace{-.05in}\left.\left(\mathbb{E}\{x[k]|(y[t])_{t=0}^k,(T^j[t])_{t=0}^k\}\hspace{-.03in}-\hspace{-.03in}\mathbb{E}\{x[k]|(y[t])_{t=0}^k\}\right)^\top\right\}\hspace{-.04in}=\hspace{-.04in}0,
\end{align*}
and, as a result,
\begin{align*}
&\mathbb{E}\left\{\|\hat{x}^{\mathrm{appx}}[k]-\hat{x}[k]\|_2^2|(y[t])_{t=0}^k\right\}\\
&=\trace\left(\mathbb{E}\left\{(\hat{x}^{\mathrm{appx}}[k]-\hat{x}[k])(\hat{x}^{\mathrm{appx}}[k]-\hat{x}[k])^\top|(y[t])_{t=0}^k\right\}\right)\\
&=\frac{1}{I}\trace\left(\mathbb{E}\left\{(\mathbb{E}\{x[k]|(y[t])_{t=0}^k,(T[t])_{t=0}^k\}-\hat{x}[k])\right.\right.\\
&\hspace{.3in}\left.\left.\times(\mathbb{E}\{x[k]|(y[t])_{t=0}^k,(T[t])_{t=0}^k\}-\hat{x}[k])^\top|(y[t])_{t=0}^k\right\}\right)\\
&=\frac{1}{I}\mathbb{E}\left\{\|\mathbb{E}\{x[k]|(y[t])_{t=0}^k,(T[t])_{t=0}^k\}-\hat{x}[k]\|_2^2|(y[t])_{t=0}^k\right\}.
\end{align*}
This gives that
\begin{align*}
&\mathbb{E}\{\|\hat{x}^{\mathrm{appx}}[k]-\hat{x}[k]\|_2^2\}\\
&=
\mathbb{E}\{\mathbb{E}\{\|\hat{x}^{\mathrm{appx}}[k]-\hat{x}[k]\|_2^2|(y[t])_{t=0}^k\}\}\\
&=\frac{1}{I}\mathbb{E}\{\mathbb{E}\{\|\mathbb{E}\{x|(y[t])_{t=0}^k,(T[t])_{t=0}^k\}-\hat{x}[k]\|_2^2|(y[t])_{t=0}^k\}\}\\
&=\frac{1}{I}\mathbb{E}\{\|\mathbb{E}\{x|(y[t])_{t=0}^k,(T[t])_{t=0}^k\}-\hat{x}[k]\|_2^2\}.
\end{align*}
Let $M[k]:=\mathbb{E}\left\{\|\mathbb{E}\{x|(y[t])_{t=0}^k,(T[t])_{t=0}^k\}\hspace{-.03in}-\hspace{-.03in}\hat{x}[k]\|_2^2\right\}$. Therefore, it can be proved that
\begin{align*}
\mathbb{P}\left\{\left\|\hat{x}^{\mathrm{appx}}[k]-\hat{x}[k]\right\|_2\leq \epsilon\right\}\geq 1-\frac{M[k]}{I\epsilon^2}.
\end{align*}
This concludes the proof.
\end{proof}

Note that the memory required for implementing the filter might grow with time. This is because if $M[k]$ grows the with time, more and more Kalman filters are required to maintain the probability $\mathbb{P}\left\{\left\|\hat{x}^{\mathrm{appx}}[k]-\hat{x}[k]\right\|_2\leq \epsilon\right\}$ the same as in the previous time step. To add a new Kalman filter, all the measurements from the past $(y[t])_{t=0}^k$ are required. Thus, the complexity of implementing the filter might grow with time as the estimator might need more Kalman filters in each iteration. The following theorem shows that, for a family of systems, $M[k]$ has a constant upper bound and, thus, it is not required to increase the number of the Kalman filters with time.

\begin{corollary} If $\rho(A)<1$, there exists a finite $M>0$ such that $\mathbb{P}\left\{\left\|\hat{x}[k]-\hat{x}^{\mathrm{appx}}[k]\right\|_2\leq \epsilon\right\}\geq 1-\delta$ for $I\geq M\delta^{-1}\epsilon^{-2}$.
\end{corollary}

\begin{proof}
Notice that
\begin{align*}
\mathbb{E}\{\|\mathbb{E}&\{x[k]|(y[t])_{t=0}^k,(T[t])_{t=0}^k\}\hspace{-.03in}-\hspace{-.03in}\hat{x}[k]\|_2^2\}\\
\leq &\mathbb{E}\left\{\|\mathbb{E}\{x[k]|(y[t])_{t=0}^k,(T[t])_{t=0}^k\}\|_2^2\right\}
+
\mathbb{E}\left\{\|\hat{x}[k]\|_2^2\right\}\\
= &\mathbb{E}\left\{\|\mathbb{E}\{x[k]|(y[t])_{t=0}^k,(T[t])_{t=0}^k\}\|_2^2\right\}\\
&+
\mathbb{E}\left\{\|\mathbb{E}\{x[k]\,|\,(y[t])_{t=0}^k\}\|_2^2\right\}\\
\leq &2\mathbb{E}\left\{\|x[k]\|_2^2\right\}\\
\leq &2\lim_{T\rightarrow\infty}\mathbb{E}\left\{\|x[k]\|_2^2\right\}.
\end{align*}
Selecting $M=2\lim_{T\rightarrow\infty}\mathbb{E}\left\{\|x[k]\|_2^2\right\}$ (which is bounded due to the fact that $A$ is stable) concludes the proof.
\end{proof}

\subsection{Generating Samples}
The Bayes' rule implies that
\begin{align*}
p(T[k]|(y[t])_{t=0}^k)
\hspace{-.03in}\propto &p(y[k]|T[k],(y[t])_{t=0}^{k-1}) p(T[k]|(y[t])_{t=0}^{k-1})\\
\propto &p(y[k]|T[k],(y[t])_{t=0}^{k-1})\\
&\times p((y[t])_{t=0}^{k-1}|T[k])p(T[k])\\
= &p(y[k]|T[k],(y[t])_{t=0}^{k-1})\\
&\times p((y[t])_{t=0}^{k-1})p(T[k])\\
\propto &p(y[k]|T[k],(y[t])_{t=0}^{k-1})p(T[k]),
\end{align*}
where $\propto$ is the proportionality operator 
(i.e., it denotes the fact that an appropriate factor should be multiplied by the right-hand side to ensure that $p(T[k]|(y[t])_{t=0}^k)$ integrates to one for each $(y[t])_{t=0}^k$). 
It can be seen that computing $p(T[k]|(y[t])_{t=0}^k)$ requires the knowledge of $p(y[k]|T[k],(y[t])_{t=0}^{k-1})$, which can only be calculated by knowing $p(x[k]|(y[t])_{t=0}^{k-1})$. Noting that computing $p(x[k]|(y[t])_{t=0}^{k-1})$ is more difficult than solving Problem~\ref{prob:1} (i.e., it makes the sampling of $T[k]$ as cumbersome as the MAP filter in Section~\ref{sec:MLE} to implement), a number of heuristics are proposed in what follows to generate samples that, at least, closely follow this distribution\footnote{Note that for deadbeat systems (i.e., $A[k]=A$ such that $A^L=0$ for some $L\in\mathbb{N}$), $p(x[k]|(y[t])_{t=0}^{k-1})$ is equal to $p(x[k]|(y[t])_{t=k-L}^{k-1})$ and, thus, a rolling window estimation can be utilized. However, for general systems, such treatments provide nice heuristics.}. 

\subsubsection{Memory-Less Generation} \label{subsec:memory-less}
This method relies on the heuristic that if $\rho(A)\ll 1$, the state of the systems is loosely correlated in time and, thus,
\begin{align*}
p(T[k]|(y[t])_{t=0}^k)\approx p(T[k]|y[k]).
\end{align*}
Now, for each $\tau_i[k]$ from $T[k]=\diag(\tau_1[k]^2,\dots,\tau_p[k]^2)$, it can be observed that
\begin{align*}
p(\tau_i[k]|y[k])
\propto & p(y[k]|\tau_i[k])p(\tau_i[k])
\propto p(y_i[k]|\tau_i[k])p(\tau_i[k]).
\end{align*}
This density function can be easily calculated noting that the state stays Gaussian in the absence of any measurements. That is, $x[k]\sim \mathcal{N}(0,X[k])$, where
\begin{align*}
X[k]=A[k]X[k]A[k]^\top + W[k], \quad X[0]=X_0.
\end{align*}
Thus,
\begin{align}
p(\tau_i[k]|y[k])\nonumber
\propto & p(y_i[k]|\tau_i[k])p(\tau_i[k])\\
\propto & \frac{1}{C_i[k]X[k]C_i[k]^\top +\tau_i[k]^2}\nonumber\\
&\times \exp\bigg(-\frac{y_i[k]^2}{2(C_i[k]X[k]C_i[k]^\top +\tau_i[k]^2)} \bigg)\nonumber\\
&\times \frac{2\tau_i[k]}{V_{ii}[k]}\exp\bigg(
-\frac{\tau_i[k]^2}{V_{ii}[k]} \bigg),
\label{eqn:long_f}
\end{align}
where $C_i[k]$ is the $i$-th row of the matrix $C[k]$. Rejection sampling~\cite{casella} can be used to generate samples that follow this density function. Let $f(\tau_i[k],y[k])$ denote the right-hand side of~\eqref{eqn:long_f}. First, generate a Rayleigh variable $\tau_i[k]\sim\mathcal{R}(\sqrt{V_{ii}[k]/2})$ and a uniform random variable $u$ between zero and one. The random variable $\tau_i[k]$ follows the density function $p(\tau_i[k]|y[k])$ if $u\leq f(\tau_i[k],y[k])/[2\tau_i[k]/V_{ii}[k]\exp(-\tau_i[k]^2/V_{ii}[k])]/\varpi$, where
\begin{align*}
\varpi
:=&\max_{\tau} \frac{1}{C_i[k]X[k]C_i[k]^\top +\tau^2}\\&\hspace{.3in}\times\exp\bigg(\frac{-y_i[k]^2}{2(C_i[k]X[k]C_i[k]^\top +\tau^2)} \bigg).
\end{align*}
Otherwise, repeat this procedure. 

\subsubsection{Aggregation over Time}
In this case, a whole batch of $T:=(T[t])_{t=0}^k$ is generated in each time step for the observed measurements $y:=(y[t])_{t=0}^k$. Define $x:=(x[t])_{t=0}^k$. Similarly, the Bayes' rule dictates that
\begin{align*}
p(T|y)
\propto  p(y|T)p(T)
= & \bigg[\int p(y|x,T)p(x|T)\mathrm{d}x\bigg] p(T)\\
= & \bigg[\int p(y|x,T)p(x)\mathrm{d}x\bigg] p(T).
\end{align*}
With the same procedure as in the previous case, one can generate samples that follow the density function  $p(T|y)$. However, the size of the problem grows unbound, which makes this approach infeasible  in practice. To be able to overcome the computational complexity and memory requirements, a moving horizon can be used while assuming that the output measurements outside of the window are loosely correlated with those  inside. 

\subsubsection{Gaussian Approximation}
If the memory-less generation and the aggregate approach are not suitable, the density function $p(x[k]|(y[t])_{t=0}^{k-1})$ can be approximated by a Gaussian using the optimal linear estimators that are constructed in Section~\ref{sec:linear}. In this case, it can be assumed that
\begin{align*}
p(x[k]|(y[t])_{t=0}^{k-1})\hspace{-.04in}\propto\hspace{-.04in} \exp\hspace{-.04in}\bigg(\hspace{-.06in}-\hspace{-.02in}\frac{1}{2}(x-\mu[k])^\top \hat{X}[k]^{-1}(x-\mu[k])\hspace{-.03in}\bigg)\hspace{-.02in}.
\end{align*} 
Now, it can be proved that
\begin{align*}
p(\tau_i[k]|(y[t])_{t=0}^k)
\propto &p(y[k]|\tau_i[k],(y[t])_{t=0}^{k-1})p(\tau_i[k])\\
\propto &p(y_i[k]|\tau_i[k],(y[t])_{t=0}^{k-1})p(\tau_i[k])\\
\propto &\frac{1}{C_i[k]\hat{X}[k]C_i[k]^\top +\tau_i[k]^2}\nonumber\\
&\times\hspace{-.03in}\exp\hspace{-.04in}\bigg(\hspace{-.05in}-\hspace{-.02in}\frac{y_i[k]^2}{2(C_i[k]\hat{X}[k]C_i[k]^\top\hspace{-.03in}+\hspace{-.03in}\tau_i[k]^2)}\hspace{-.03in} \bigg)\\
&\times \frac{2\tau_i[k]}{V_{ii}[k]}\exp\bigg(
-\frac{\tau_i[k]^2}{V_{ii}[k]} \bigg).
\end{align*}
Similarly, the rejection sampling can be utilized to generate samples that follow a density of this form.

\section{Optimal Linear Estimators} \label{sec:linear}
In this section, optimal linear estimators are investigated. This filter is only presented in the paper to assist with their comparison with the estimator developed using the randomized algorithm. 
It is well-known that the Kalman filter is the optimal linear estimator (i.e., it achieves the least variance among all the linear estimators); see, for example,~\cite[p.\,49]{krishnamurthy2016partially} and~\cite[p.\,10]{maybeck1982stochastic}. Therefore, the optimal linear estimator is 
\begin{align}
\hat{x}[k]=
&A[k-1]\hat{x}[k-1]\nonumber\\
&+L[k](y[k]\hspace{-.03in}-\hspace{-.03in}C[k]A[k-1]\hat{x}[k-1]), \, \hat{x}[0]=0,
\end{align}
where 
\begin{align*}
L[k]=&\bar{P}[k]C[k]^\top(V[k]+C[k]\bar{P}[k]C[k]^\top)^{-1},\\
\bar{P}[k]=&A[k]\check{P}[k-1]A[k]^\top+W[k],\\
\check{P}[k]
=&\bar{P}[k]-L[k]C[k]\bar{P}[k]-\bar{P}[k]C[k]^\top L[k]^\top\\
&+L[k](C[k]\bar{P}[k]C[k]^\top +V[k])L[k]^\top.
\end{align*}
It is worth noting that $\check{P}[k]=\mathbb{E}\{(x[k]-\hat{x}[k])(x[k]-\hat{x}[k])^\top\}$.
For constant model parameters $A[k]=A$, $B[k]=B$, $C[k]=C$, $W[k]=W$, and $V[k]=V$, the optimal linear estimator has a bounded variance if $(A,C)$ is observable. 

\section{MAP Estimator} \label{sec:MLE}
This section follows the same approach as in~\cite{carrillo2015state} to construct the maximum \textit{a posteriori} estimate of the state. To do so, for a given $k\in\mathbb{N}$, define
\begin{align*}
x=\begin{bmatrix}
 x[1] \\ \vdots \\ x[k]
\end{bmatrix}\in\mathbb{R}^{nk},\hspace{.1in} 
w=\begin{bmatrix}
 w[0] \\ \vdots \\ w[k-1]
\end{bmatrix}\in\mathbb{R}^{nk}.
\end{align*}
It can be seen that
\begin{align*}
x=\Omega x_0+\Psi w,
\end{align*}
where
\begin{align*}
\Omega\hspace{-.03in}=\hspace{-.05in}\begin{bmatrix}
A[0] \\ \vdots \\ A[k-1]\cdots A[0]
\end{bmatrix},
\Psi\hspace{-.03in}=\hspace{-.05in}\begin{bmatrix}
I & \cdots & 0 \\ \vdots & \ddots & \vdots \\ A[k-1]\cdots A[1] & \cdots & I
\end{bmatrix}\hspace{-.03in}.
\end{align*}
Hence, $x\sim \mathcal{N}(0,\Omega X_0 \Omega^\top+\Psi W\Psi^\top)$, where $W=\diag(W[0],\dots,W[k-1]).$
Further, define
\begin{align*}
y=\begin{bmatrix}
 y[0] \\ \vdots \\ y[k]
\end{bmatrix}\in\mathbb{R}^{p(k+1)},\hspace{.1in} 
v=\begin{bmatrix}
 v[0] \\ \vdots \\ v[k]
\end{bmatrix}\in\mathbb{R}^{p(k+1)}.
\end{align*}
Note that
\begin{align*}
y=Cx+v,
\end{align*}
where $C=\diag(C[0],\dots,C[k-1]).$ Note that $v_i\sim \mathcal{L}(0,\sqrt{V_{ii}/2})$, where $V=\diag(V[0],\dots,V[k]).$
Therefore, the conditional density of $y$ given $x$ is given by
\begin{align*}
p(y|x)
=&\prod_{i=1}^{p(k+1)}\frac{1}{\sqrt{2V_{ii}}}\exp\left(-\frac{\sqrt{2}|y_i-(Cx)_i|}{\sqrt{V_{ii}}}\right)
\\
=&\bigg(\prod_{i=1}^{p(k+1)}\frac{1}{\sqrt{2V_{ii}}}\bigg)\exp\left(-\sqrt{2}\sum_{i=1}^{p(k+1)}\frac{|y_i-(Cx)_i|}{\sqrt{V_{ii}}}\right).
\end{align*}
Now, the Bayes' rule~\cite{durrett2010probability} can be used to find the conditional density function of $x$ given $y$ as
\begin{align*}
p(x|y)
\propto 
&p(y|x)\bigg[\frac{1}{\sqrt{(2\pi)^{nk}\det(X)}}\exp\bigg(-\frac{1}{2}x^\top X^{-1}x \bigg) \bigg]\\
=&\frac{1}{\sqrt{\prod_{i=1}^{p(k+1)}(2V_{ii})}\sqrt{(2\pi)^{nk}\det(X)}}\\
&\times\exp\bigg(-\frac{1}{2}x^\top X^{-1}x-\sqrt{2}\sum_{i=1}^{p(k+1)}\frac{|y_i-(Cx)_i|}{\sqrt{V_{ii}}} \bigg)
\end{align*}
where $X=\Omega X_0 \Omega^\top+\Psi W\Psi^\top$.
Therefore, the MAP estimate of the state at time $k$ is given by
\begin{align*}
\hat{x}[k]=
\begin{bmatrix}
0_{n\times (k-1)n} & I_n
\end{bmatrix}x^{(k)*},
\end{align*}
where
\begin{align*}
x^{(k)*}\in\argmin_{x\in\mathbb{R}^{nk}}
\|\Phi_1 x\|_2^2+\|\Phi_1(y-Cx)\|_1,
\end{align*}
with $\Phi_1=(1/\sqrt{2})X^{-1/2}$ and $\Phi_2=\sqrt{2}\diag(V_{11}^{-1/2},\dots,V_{pp}^{-1/2})$. The memory required for this filter also grows with $k$ because one needs to remember the whole measurement history. For practical purposes, the filter can be implemented on a rolling window in which case the prior on the  state at the beginning of the rolling window can be approximated with a Gaussian distribution. 

\subsection{Scalar Measurements and Window of Length One}
This subsection deals with the relatively simple, yet insightful, case of a window of length one for scalar measurements. In this case, the filter can be explicitly calculated. Assume that the distribution of $x[k]$ given the measurements, i.e., $p(x[k]|(y[t])_{t=0}^{k})$, can be approximated by a Gaussian distribution with mean $\mu[k]$ and variance $\Xi[k]$. Therefore, the prior on $x[k+1]$ is best approximated by a Gaussian distribution with mean $\mu'[k+1]:=A[k]\mu[k]$ and the variance $\Xi'[k+1]:=A[k]^\top \Xi[k]A[k]+W[k]$. Similarly, it can be proved that
\begin{align*}
p(x&[k+1]|(y[t])_{t=0}^{k+1})
\\ \propto & p(y[k+1]|x[k+1])
 p(x[k+1]|(y[t])_{t=0}^{k})\\
\approx & \exp\bigg(\hspace{-0.06in}-\frac{1}{2}\|\Xi'[k+1]^{-1/2}(x[k+1]-\mu'[k+1])^\top\|_2^2 \\
& -\sqrt{\frac{2}{V[k+1]}}\|C[k+1]x[k+1]-y[k+1]\|_1 \bigg).
\end{align*}
Therefore, the windowed MAP estimate at $k$ is given by
\begin{align} \label{eqn:W1MAP}
\hat{x}[k+1]\in \argmin_{x\in\mathbb{R}^n} \bigg(-\frac{1}{2}&\|\Xi'[k+1]^{-1/2}(x-\mu'[k+1])\|_2^2
\nonumber\\ &\hspace{-1.15in}-\sqrt{\frac{2}{V[k+1]}}\|C[k+1]x[k+1]-y[k+1]\|_1\bigg).
\end{align}
Setting the derivative of the cost function of~\eqref{eqn:W1MAP} with respect to $x$ equal to zero, for $C[k+1]x>y[k+1]$, gives that
\begin{align*}
x=\mu'[k+1]-\sqrt{\frac{2}{V[k+1]}}\Xi'[k+1]C[k+1]^\top.
\end{align*}
However, the same calculations for $C[k+1]x<y[k+1]$ results in
\begin{align*}
x=\mu'[k+1]+\sqrt{\frac{2}{V[k+1]}}\Xi'[k+1]C[k+1]^\top.
\end{align*}
For the rest of the cases, it can be deduced that $C[k+1]x=y[k+1]$. Therefore, 
$x=(C[k+1]^\top C[k+1])^{-1}C[k+1]^\top y[k+1]+N[k+1]\alpha,$
where the columns of the matrix $N[k+1]$ form an orthonormal basis for the null space of $C[k+1]$ and $\alpha\in\mathbb{R}^{n-1}$ is an arbitrary vector. In this case, substituting $x$ into~\eqref{eqn:W1MAP} gives
\begin{align*}
\alpha^*\hspace{-.03in}\in\hspace{-.03in} \argmin_{\alpha\in\mathbb{R}^{n-1}} -\frac{1}{2}&\|\Xi'[k\hspace{-.03in}+\hspace{-.03in}1]^{-1/2}(x'\hspace{-.03in}+\hspace{-.03in}N[k\hspace{-.03in}+\hspace{-.03in}1]\alpha\hspace{-.03in}-\hspace{-.03in}\mu'[k\hspace{-.03in}+\hspace{-.03in}1])^\top\hspace{-.03in}\|_2^2,
\end{align*}
where $x'=(C[k+1]^\top C[k+1])^{-1}C[k+1]^\top y[k+1]$.
By setting the derivative of the cost function of this optimization problem with respect to $\alpha$ equal to zero, it can be shown that $x'+N[k+1]\alpha^*-\mu'[k+1]=0,$
and, as a result,
$\alpha^*=N[k+1]^\top \mu[k+1]'.$
Therefore, the windowed MAP estimate at $k$ is given by~\eqref{eqn:windowed_MAP}, where
\begin{align*}
\zeta[k]=\sqrt{2/V[k+1]}(A[k]^\top \Xi[k]A[k]+W[k])C[k+1]^\top.
\end{align*}
\begin{figure*}
\begin{align}
\hat{x}[k+1]=
\begin{cases}
A[k]\mu[k]-\zeta[k], & \hspace{-2.3in} y[k+1]<C[k+1]A[k]\mu[k]-C[k+1]\zeta[k],\\
N[k+1]^\top A[k]\mu[k]+(C[k+1]^\top C[k+1])^{-1}C[k+1]^\top y[k+1], &\\
& \hspace{-2.3in} C[k+1]A[k]\mu[k]-C[k+1]\zeta[k]\leq y[k+1]\leq C[k+1]A[k]\mu[k]+C[k+1]\zeta[k],\\
A[k]\mu[k]+\zeta[k], & \hspace{-2.3in}C[k+1]A[k]\mu[k]-C[k+1]\zeta[k]<y[k+1].
\end{cases} \label{eqn:windowed_MAP}
\end{align}
\hrulefill
\vspace{-.2in}
\end{figure*}
Using the closed-form expression in~\eqref{eqn:windowed_MAP}, $\mu[k+1]=\mathbb{E}\{x[k+1]\}$ and $\Xi[k+1]=\mathbb{E}\{(x[k+1]-\mu[k+1])(x[k+1]-\mu[k+1])^\top\}$ can be calculated. Note that, although requiring a finite memory as in the case of the Kalman filter, the updates of the filter are nonlinear.

\section{Numerical Example} \label{sec:numerical}
As a numerical example, consider the discrete-time linear time-invariant dynamical system in
\begin{align*} 
x[k+1]&=\begin{bmatrix}
0.9 & 1.0 \\ 0.0 & 0.8
\end{bmatrix}x[k]+w[k],\,x[0]=0,\\
y[k]&=\begin{bmatrix}
1.0 & 0.0
\end{bmatrix}
x[k]+v[k],
\end{align*}
with $W=\diag(1.0,1.5)$ and $V=10$. Note that the pair $(A,C)$ is observable. The previously mentioned methods can be utilized to construct an estimate of the state.

Fig.~\ref{fig:1} illustrates the estimation error  $\mathbb{E}\{\|\hat{x}[k]-x[k]\|_2^2\}$ as a function of time $k$ for the some of the developed estimators. Firstly, the black curve in Fig.~\ref{fig:1} shows the error of the optimal linear estimator, which is equal to $\trace(\check{P}[k])$. Now, note that
\begin{align}
&\mathbb{E}\{\|\hat{x}^{\mathrm{appx}}[k]-x[k]\|_2^2\}\nonumber\\
&=\mathbb{E}\bigg\{\bigg\|\frac{1}{I}\sum_{i=1}^I \mathbb{E}\{x[k]|(y[t])_{t=0}^k,(T^i[t])_{t=0}^k\}-x[k]\bigg\|_2^2\bigg\}\nonumber\\
&= \mathbb{E}\bigg\{\frac{1}{I}\sum_{i=1}^I \|\mathbb{E}\{x[k]|(y[t])_{t=0}^k,(T^i[t])_{t=0}^k\}-x[k]\|_2^2\bigg\}\label{eqn:proof:ident}\\
&= \mathbb{E}\bigg\{\frac{1}{I}\sum_{i=1}^I \mathbb{E}\{\|\mathbb{E}\{x[k]|(y[t])_{t=0}^k,(T^i[t])_{t=0}^k\}\nonumber\\
&\hspace{1.7in}-x[k]\|_2^2|(T^i[t])_{t=0}^k\}\bigg\},\nonumber
\end{align}
where the equality in~\eqref{eqn:proof:ident} follows from that the zero-mean random variables $\mathbb{E}\{x[k]|(y[t])_{t=0}^k,(T^i[t])_{t=0}^k\}-x[k]$ and $\mathbb{E}\{x[k]|(y[t])_{t=0}^k,(T^j[t])_{t=0}^k\}-x[k]$ are independent  if $i\neq j$. Let $\tilde{P}^i[k]$ denote the covariance of the error corresponding to the Kalman filter with $(T^i[t])_{t=0}^k$. Thus,
\begin{align}
\mathbb{E}\{\|\hat{x}^{\mathrm{appx}}[k]-x[k]\|_2^2\}&=\mathbb{E}\bigg\{\frac{1}{I}\sum_{i=1}^I \trace(\tilde{P}^i[k])\bigg\}.
\label{eqn:righthand}
\end{align}
The solid red curve in Fig.~\ref{fig:1} shows $(1/I)\sum_{i=1}^I \trace(\tilde{P}^i[k])$ for $I=1000$ as an approximation of $\mathbb{E}\{\|\hat{x}^{\mathrm{appx}}[k]-x[k]\|_2^2\}$ when using the memory-less heuristic for generating $T^i[k]$. This estimator is clearly the best in terms of the performance, however, it has a higher computational load. The green curve in Fig.~\ref{fig:1} illustrates a Monte Carlo approximation of $\mathbb{E}\{\|\hat{x}[k]-x[k]\|_2^2\}$ for the MAP estimate with $10000$ scenarios. The MAP estimator clearly performs worse in terms of the variance of the error. However, such a comparison is not entirely fair for the MAP estimator because it is not designed to minimize the variance of the error as opposed to the other filters. Finally, the dotted black curve shows the variance of the error of the particle filter with 1000 particles (for a fair comparison with the optimal least mean square error estimator approximated using the method of Section~\ref{sec:randomized}). The variance is constructed using a Monte Carlo approximation with $10000$ scenarios. As we can see, the particle filter performs slightly worse. This is because the particle filter does not fully utilizes the specific structure of the problem (i.e., the fact that the system is linear and noise is Laplace). This makes the particle filter a versatile tool for a vast majority of problem; however, it makes it slightly more conservative in specific cases. To run the particle filter, we have used the toolbox in~\cite{particlefiltertoolbox}, which is based on~\cite{arulampalam2002tutorial}.

\begin{figure}
\centering
\begin{tikzpicture}
\node[] at (0,0) {\includegraphics[width=1.0\linewidth]{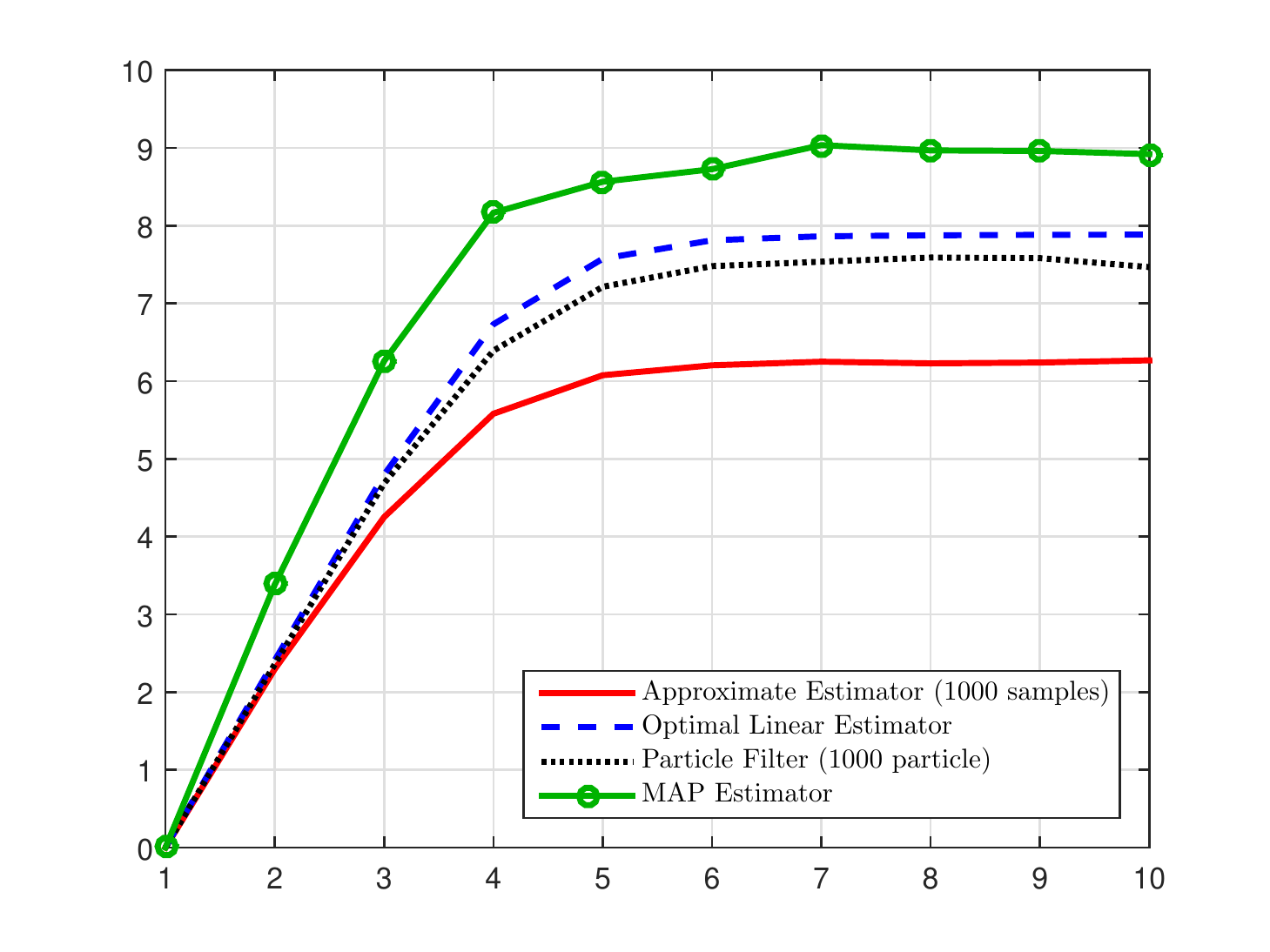}};
\node[] at (0,-3.1) {$k$};
\node[rotate=90] at (-3.8,0) {$\mathbb{E}\{\|\hat{x}[k]-x[k]\|_2^2\}$};
\end{tikzpicture}
\vspace*{-.3in}
\caption{\label{fig:1} The estimation error $\mathbb{E}\{\|\hat{x}[k]-x[k]\|_2^2\}$ as a function of time $k$ for the optimal linear estimator (dashed blue\,\ole), the optimal least mean square error estimator approximated using the method of Section~\ref{sec:randomized} with $I=1000$ (solid red\,\approximation), the particle filter with $1000$ particles (dotted black\,\particle), and the MAP estimator (green\,\map). }
\vspace*{-.2in}
\end{figure}

\section{Conclusions and Future Work} \label{sec:conc}
Optimal state estimation for linear discrete-time systems with measurements corrupted by Laplace noise was considered. A randomized method was used to approximate the optimal least mean square error estimate of the state. This method also works for any other noise that can be decomposed as a multiplication of stochastic variable with an independently drawn Gaussian noise (e.g., Cauchy or normal-product distributed noise). 
Future research can focus on finding a tighter bound on the probability of the event that the error between the approximate state estimate using the randomized method and the optimal least mean square error estimate is small. A lower bound for this probability was constructed here using the Chebyshev's inequality, which does not take into account the special distribution of the state in this problem formulation. Further work can be dedicated to constructing efficient algorithms for sampling in the randomized method.

\bibliographystyle{IEEEtran}
\bibliography{citation}

\end{document}